\documentclass[a4paper,12pt]{article}
\usepackage{comment}
\usepackage{cite}
\usepackage{amsmath}
\usepackage{amssymb}
\usepackage{amsfonts}
\usepackage[T1]{fontenc}
\usepackage[utf8]{inputenc}
\usepackage{graphicx}
\usepackage{fancyhdr}
\usepackage{float}
\usepackage{xcolor}
\usepackage{authblk}
\usepackage{mathrsfs}
\usepackage{empheq}
\usepackage[hyphens]{url}
\usepackage{hyperref}
\usepackage[]{breakurl}
\usepackage[]{amsthm}
\usepackage{tikz}
\usetikzlibrary{decorations.markings}

\usepackage{geometry}
\newtheorem{theorem}{Theorem}

\newcommand{\nobarfrac}{\genfrac{}{}{0pt}{}}

\newcommand{\stirdeux}[2]{\left\{\nobarfrac{#1}{#2}\right\}}

\newcommand{\Eulerian}[2]{%
\mathinner{%
\genfrac\langle\rangle{0pt}{}{#1}{#2}%
}%
}

\begin{document}

\title{Sum rules for permutations with fixed points involving Stirling numbers of the first kind}

\author[$\dagger$]{Jean-Christophe {\sc Pain}\footnote{jean-christophe.pain@cea.fr}\\
\small
$^1$CEA, DAM, DIF, F-91297 Arpajon, France\\
$^2$Universit\'e Paris-Saclay, CEA, Laboratoire Mati\`ere en Conditions Extr\^emes,\\
F-91680 Bruy\`eres-le-Ch\^atel, France
}

\date{}

\maketitle

\begin{abstract}
We propose sum rules for permutations $p_n(k)$ of the ensemble $\left\{1,2,\cdots,n\right\}$ with $k$ fixed points, in the form of partial sums of their moments. The corresponding identities involve Stirling numbers of the first kind $s(q,r)$. Using a formula due to Vassilev-Missana and the Schl\"omlich expression of Stirling numbers, we also deduce sum rules for binomial coefficients. Connections with Bell numbers $B_n$ are outlined.
\end{abstract}

\section{Introduction}\label{sec1}

Let us denote by $p_n(k)$ the number of permutations of the ensemble $\left\{1,2,\cdots,n\right\}$ which have exactly $k$ fixed points. A permutation has exactly one fixed point if and only if it is a derangement of $(n - 1)$ elements (fixing one element from $\left\{1,2\cdots,n\right\}$). There are $n$ choices for which element is fixed and $d_{n-1}$ derangements of the remaining elements. Hence there are $n\,d_{n-1}$ such permutations, where $d_k$ represents the number of derangements of $k$ elements. Performing a permutation with $k$ fixed points boils down to choosing $k$ fixed points among the $n$ initial ones, and then a derangement (permutation without fixed point) for the $(n-k)$ remaining ones. Thus
\begin{equation}\label{rel}
    p_n(k)=\binom{n}{k}\,d_{n-k},
\end{equation}
with $d_k=p_k(0)$. One has, following Ref. \cite{Tian2014}:
\begin{equation*}
    p_n(0)=\sum_{k=1}^n(k-1)\binom{n}{k}\,p_{n-k}(0).    
\end{equation*}
Such a relation is obtained from the Euler recurrence relation for derangements
\begin{equation*}
    p_n(0)=(n-1)\left[p_{n-1}(0)+p_{n-2}(0)\right]
\end{equation*}
and, as pointed out by the authors of Ref. \cite{Tian2014}, has first been proved by Deutsch and Elizalde \cite{Deutsch2010}, who gave a bijective proof and a proof involving generating functions. One has also, using Eq. (\ref{rel}):
\begin{equation}\label{four}
    p_{n-k}(0)=\sum_{l=1}^{n-k}(l-1)\binom{n-k}{l}\,p_{n-k-l}(0).    
\end{equation}
Thus, using Eq. (\ref{rel}), one finds
\begin{equation*}
    p_n(k)=\binom{n}{k}\,p_{n-k}(0),
\end{equation*}
yielding
\begin{equation*}
    p_{n-k-l}(0)=\frac{p_{n-k}(l)}{\displaystyle\binom{n-k}{l}}
\end{equation*}
which gives, according to Eq. (\ref{four}):
\begin{equation*}
    \frac{p_n(k)}{\displaystyle\binom{n}{k}}=\sum_{l=1}^{n-k}(l-1)\binom{n-k}{l}\,\frac{p_{n-k}(l)}{\displaystyle\binom{n-k}{l}},    
\end{equation*}
and thus
\begin{equation*}
    p_n(k)=\binom{n}{k}\sum_{l=1}^{n-k}(l-1)\,p_{n-k}(l).    
\end{equation*}
Only a limited number of identities are known for permutations that contain fixed points. Most of these identities are expressed in terms of Stirling numbers of the second kind or Bell numbers. The Stirling numbers of the second kind, denoted by $\stirdeux{q}{k}$, count the number of distinct ways to decompose a set of $n$ elements into $k$ non-empty subsets \cite{Pitman1997}. The Bell number $B_q$ is the number of partitions of a set of cardinality $q$ \cite{Bell1934}. It is also the number of $n-$pattern sequences \cite{Cooper1992}, and is related to Stirling numbers of the second kind through 
\begin{equation*}
    B_q=\sum_{k=0}^q\stirdeux{q}{k}.
\end{equation*}
Actually for every natural number $n\geq q$ one can write \cite{Wilf1994}:
\begin{equation}\label{dobi}
    B_q=\sum_{k=0}^n\stirdeux{q}{k}.
\end{equation}
Inserting the usual expression of Stirling number of the second kind
\begin{equation*}
    \stirdeux{q}{k}=\frac{1}{k!}\sum_{j=0}^k(-1)^{k-j}\binom{k}{j}j^q
\end{equation*}
into Eq. (\ref{dobi}) yields
\begin{align*}
    B_q&=\sum_{k=0}^n\frac{1}{k!}\sum_{j=0}^k(-1)^{k-j}\binom{k}{j}j^q\nonumber\\
    &=\sum_{k=0}^n\sum_{j=0}^k(-1)^{k-j}\frac{j^q}{j!(k-j)!}\nonumber\\
    &=\sum_{j=0}^n\frac{j^q}{j!}\sum_{k=j}^n\frac{(-1)^{k-j}}{(k-j)!}\nonumber\\
    &=\sum_{j=0}^n\frac{j^q}{j!}\sum_{s=0}^{n-j}\frac{(-1)^s}{s!}
\end{align*}
and since
\begin{equation*}
    p_n(j)=\frac{n!}{j!}\sum_{s=0}^{n-j}\frac{(-1)^s}{s!},
\end{equation*}
one finds
\begin{equation}\label{newdobi}    
   B_q=\frac{1}{n!}\sum_{j=0}^{n}j^q\,p_n(j).
\end{equation}
The foregoing formula is intimately tied to the fact that the limit law for the number of fixed points is a Poisson law \cite{Fulman2024}. Also it is quite easy to see a bijective interpretation of it, since it amounts to count permutations of size $n$ with a marked ordered $q-$tuple of fixed points (allowing a fixed point being repeated). For example, for $q=2$ one sees well the two cases, upon counting the number of possibilities by choosing first the values of the fixed points, and then the remaining permutation (when the two marked fixed points are different the associated partition of 2 is $\{1\},\{2\}$, when they are equal the associated partition of 2 is $\{1,2\}$).

It turns out that Eq. (\ref{newdobi}) can be interpreted as a variation of Dobi\'nski's formula \cite{Dobinski1877}: 
\begin{equation*}
    B_q=\frac{1}{e}\sum_{k=0}^q\frac{k^q}{k!}.
\end{equation*}

Note that since the number of derangements of $k$ elements reads
\begin{equation}\label{dera}
    p_k(0)=d_k=k!\sum_{i=0}^k\frac{(-1)^i}{i!}, 
\end{equation}
equation (\ref{dobi}) can be recast into
\begin{equation*}
    B_q=\sum_{k=1}^n\frac{k^qd_{n-k}}{k!(n-k)!}.
\end{equation*}
It is worth mentioning that using known upper bounds for Bell numbers, it it possible to directly derive bounds for finite sums of permutations with fixed points, or directly for $p_n(k)$, by inverting the relation (\ref{newdobi}) (see \ref{appA}).

In section \ref{sec2}, we derive a family of sum rules (or identities) involving Stirling numbers of the first kind \cite{Comtet1972}. Resulting sum rules for binomial coefficients are obtained in section \ref{sec3}.

\section{Sum rules involving Stirling numbers of the first kind}\label{sec2}

\subsection{The generating-function technique}\label{subsec21}

Let us consider
\begin{equation*}
    f_n(t)=\sum_{k=0}^nt^k\,p_n(k)
\end{equation*}
and calculate the generating function of $f_n(t)/n!$:
\begin{equation*}
    G(x)=\sum_{n=0}^{\infty}\frac{1}{n!}\left(\sum_{k=0}^nt^kp_n(k)\right)x^n.
\end{equation*}
Thanks to Eq. (\ref{rel}), one gets
\begin{align}\label{ert}
    G(x)&=\sum_{n=0}^{\infty}\frac{1}{n!}\left(\sum_{k=0}^nt^k\binom{n}{k}p_{n-k}(0)\right)x^n\nonumber\\
    &=\sum_{k,r=0}^{\infty}p_k(0)\frac{x^k}{k!}\frac{(xt)^r}{r!}\nonumber\\
    &=e^{xt}\sum_{k=0}^{\infty}\frac{p_k(0)}{k!}x^k.
\end{align}
Using the number of derangements of $k$ elements given by Eq. (\ref{dera}), Eq. (\ref{ert}) becomes
\begin{align*}
    G(x)&=e^{xt}\sum_{k=0}^{\infty}\left(\sum_{i=0}^k\frac{(-1)^i}{i!}\right)x^k\nonumber\\
    &=e^{xt}\sum_{i=0}^{\infty}\frac{(-1)^i}{i!}\sum_{k=i}^{\infty}x^k=e^{xt}\sum_{i=0}^{\infty}\frac{(-1)^i}{i!}\frac{x^i}{1-x}
\end{align*}
and thus finally
\begin{equation*}
    G(x)=\frac{e^{x(t-1)}}{1-x},
\end{equation*}
which is nothing else than the generating function of the derangements $e^{-x}/(1-x)$, multiplied by $e^{xt}$.

\subsection{Extracting the coefficients}

\begin{theorem}

For $n$ a non-zero natural number, and $m$ a natural number, one has
\begin{equation}\label{mainres}
    \sum_{k=0}^n\sum_{i=0}^{r+1}s(r+1,i)\,k^i\,p_n(k)=n!,
\end{equation}
where $p_n(k)$ is the permutation of $\left\{1,2\cdots,n\right\}$ with $k$ fixed points and $s(q,r)$ represents the (signed) Stirling number of the first kind.
\end{theorem}

\begin{proof}

Expanding $G(x)$ and identifying the coefficients of $x^n$, one finds
\begin{equation*}
    \sum_{k=0}^nt^kp_n(k)=n!\sum_{i=0}^{n}\frac{(t-1)^i}{i!}.
\end{equation*}
Evaluating the $(r-1)^{th}$ derivative at $t=0$, with $1\leq r<n-1$ yields
\begin{equation*}
    \sum_{k=0}^nk(k-1)(k-2)\cdots(k-r)p_n(k)=n!.
\end{equation*}
The descending product involved in the latter expression can be expressed in terms of monomials using
\begin{equation*}
    (x)_n=x(x-1)(x-2)\cdots(x-q+1)=\sum_{k=0}^qs(q,k)x^k,
\end{equation*}
where $s(q,k)$ is a signed Stirling number of the first kind. The sign of $s(q,k)$ is the same as the sign of $(-1)^{q-k}$ and one has in particular
\begin{equation*}
    \sum_{k=0}^q(-1)^ks(q,k)=(-1)^qq!
\end{equation*}
and $\forall k>q$, $s(q,k)=0$. The Stirling numbers of the first kind $s(q,k)$ satisfy, for $1\leq k\leq q$, the recurrence relation
\begin{equation*}
    s(q+1,k)=s(q,k-1)-q~s(q,k),
\end{equation*}
with $s(0,0)=1$ and $\forall q\geq 1$: $s(q,0)=s(0,q)=0$. Connections of Stirling numbers with combinatorics are briefly sketched in \ref{appB}. In the present case, we get
\begin{equation*}
    x(x-1)(x-2)\cdots(x-r)=\sum_{k=0}^{r+1}s(r+1,k)x^k
\end{equation*}
yielding
\begin{equation*}
    \sum_{k=0}^n\sum_{i=0}^{r+1}s(r+1,i)\,k^i\,p_n(k)=n!,
\end{equation*}
which completes the proof.

\end{proof}

The case $q=1$ was the subject of an olympiad problem \cite{IMO1987}. It can be easily obtained from Eq. (\ref{rel}) by induction. For $n=1$, one has
\begin{equation*}
    \sum_{k=0}^1k\,p_n(k)=p_1(1)=\binom{1}{1}d_0=1=1!,
\end{equation*}
and
\begin{equation*}
    \sum_{k=0}^2k\,p_n(k)=p_2(1)+2p_2(2)=\binom{2}{1}d_1+2\binom{2}{2}d_0=2=2!,
\end{equation*}
since $d_0=1$ and $d_1=0$. Let us assume that
\begin{equation*}
    \sum_{k=0}^{n-1}k\,p_{n-1}(k)=(n-1)!.
\end{equation*}
At the next step, one has
\begin{align*}
    \sum_{k=0}^nk\,p_n(k)&=\sum_{k=0}^nk\binom{n}{k}p_{n-k}(0)\nonumber\\
    &=n\sum_{k=0}^n\binom{n-1}{k-1}p_{n-1-(k-1)}(0)\nonumber\\
    &=n\times(n-1)!=n!.
\end{align*}
The next values are
\begin{equation*}
    \sum_{k=0}^nk^2p_n(k)=2\,n!, \qquad\sum_{k=0}^nk^3p_n(k)=5\,n!,
\end{equation*}
and
\begin{equation*}
    \sum_{k=0}^nk^4p_n(k)=15\,n!, \qquad\sum_{k=0}^nk^5p_n(k)=52\,n!.
\end{equation*}

\section{Binomial sum rules}\label{sec3}

By employing identity (\ref{mainres}) alongside a formula by Vassilev-Missana and the Schlömilch expression for Stirling numbers, this section establishes several sum rules for binomial coefficients.

In 2005, Vassilev-Missana obtained the following identity \cite{Howard1974,Vassilev2005} (see \ref{appC}):
\begin{equation}\label{vassi}
    k^i=\sum_{l=0}^{\lfloor\frac{(k-1)i}{k}\rfloor}(-1)^l\binom{i}{l}\binom{k(i-l)}{i},
\end{equation}
where $\lfloor x\rfloor$ represents the integer part of $x$. Equation (\ref{vassi}), combined with Eq. (\ref{mainres}) yields
\begin{equation}\label{mainres2}
    \sum_{k=0}^n\sum_{i=0}^{r+1}\sum_{l=0}^{\lfloor\frac{(k-1)i}{k}\rfloor}s(r+1,i)(-1)^l\binom{i}{l}\binom{k(i-l)}{i}\,p_n(k)=n!.
\end{equation}
Inserting the Schl\"omlich expression \cite{Schlomlich1852,Comtet1974}:
%
%
\begin{equation*}
    s(r+1,i)=\sum_{0\leq j\leq h\leq r+1-i}(-1)^{j+h}\binom{h}{j}\binom{r+h}{r+1-i+h}\binom{2r+2-i}{r+1-i-h}\frac{(h-j)^{r+1-i+h}}{h!}
\end{equation*}
into Eq. (\ref{mainres2}), yields
\begin{align*}
    \sum_{k=0}^n\sum_{i=0}^{r+1}\sum_{l=0}^{\lfloor\frac{(k-1)i}{k}\rfloor}\sum_{0\leq j\leq h\leq r+1-i}&(-1)^{j+h}\binom{h}{j}\binom{r+h}{r+1-i+h}\binom{2r+2-i}{r+1-i-h}\frac{(h-j)^{r+1-i+h}}{h!}\\
    &\times (-1)^l\binom{i}{l}\binom{k(i-l)}{i}\,p_n(k)=n!.
\end{align*}
We can in turn use Eq. (\ref{vassi}) to replace $(h-j)^{r+1-i+h}$ by the binomial sum (\ref{vassi}) in the latter expression: 
\begin{equation*}
    (h-j)^{r+1-i+h}=\sum_{t=0}^{\lfloor\frac{(h-j-1)(r+1+h-i)}{h-j}\rfloor}(-1)^t\binom{r+1+h-i}{t}\binom{(h-j)(r+1+h-i-t)}{r+1-i+h}
\end{equation*}
yielding
\begin{align*}
    \sum_{k=0}^n\sum_{i=0}^{r+1}\sum_{l=0}^{\lfloor\frac{(k-1)i}{k}\rfloor}&\sum_{0\leq j\leq h\leq r+1-i}\sum_{t=0}^{\lfloor\frac{(h-j-1)(r+1+h-i)}{h-j}\rfloor}(-1)^{t+j+h+l}\binom{r+1+h-i}{t}\nonumber\\
    &\binom{(h-j)(r+1+h-i-t)}{r+1-i+h}\binom{h}{j}\binom{r+h}{r+1-i+h}\binom{2r+2-i}{r+1-i-h}\nonumber\\
    &\times\binom{i}{l}\binom{k(i-l)}{i}\,p_n(k)=n!,
\end{align*}
and in a second step, since
\begin{equation}\label{expresspnk}
    p_n(k)=\frac{n!}{k!}\sum_{m=0}^{n-k}\frac{(-1)^m}{m!},
\end{equation}
one gets
\begin{align*}
    \sum_{k=0}^n\sum_{m=0}^{n-k}\sum_{i=0}^{r+1}\sum_{l=0}^{\lfloor\frac{(k-1)i}{k}\rfloor}&\sum_{0\leq j\leq h\leq r+1-i}\sum_{t=0}^{\lfloor\frac{(h-j-1)(r+1+h-i)}{h-j}\rfloor}(-1)^{t+j+h+l}\binom{r+1+h-i}{t}\nonumber\\
    &\binom{(h-j)(r+1+h-i-t)}{r+1-i+k}\binom{h}{j}\binom{r+h}{r+1-i+h}\binom{2r+2-i}{r+1-i-h}\nonumber\\
    &\times\binom{i}{l}\binom{k(i-l)}{i}\,\frac{1}{k!}\frac{(-1)^m}{m!}=1.
\end{align*}
Note that using the following representation of Bell numbers:
\begin{equation*}
    B_q=\sum_{k=1}^q\sum_{i=1}^k\frac{(-1)^{k-i}\,i^q}{k!}
\end{equation*}
combined with Eq. (\ref{vassi}) again, one gets
\begin{equation*}
    B_q=\sum_{k=1}^q\sum_{i=1}^k\sum_{l=0}^{\lfloor\frac{(i-1)q}{i}\rfloor}\frac{(-1)^{l+k-i}}{k!}\binom{q}{l}\binom{i(q-l)}{q}  
\end{equation*}
which, for $n\geq q$, is equal to (see Eq. (\ref{newdobi})):
\begin{equation*}
    \frac{1}{n!}\sum_{k=0}^nk^q\,p_n(k).
\end{equation*}
Replacing $k^q$ and $p_n(k)$ by their expressions given in Eqs. (\ref{vassi}) and (\ref{expresspnk}) respectively, leads to
\begin{equation*}
    B_q=\sum_{k=0}^n\sum_{l=0}^{\lfloor\frac{(k-1)q}{k}\rfloor}\sum_{i=0}^{n-k}\frac{(-1)^{l+i}}{k!i!}\binom{q}{l}\binom{k(q-l)}{q}.
\end{equation*}
Similar identities can be obtained, for instance using the Worpitzky identity \cite{Worpitzky1883,Comtet1974}:
\begin{equation*}
    k^i=\sum_{j=0}^i\Eulerian{i}{j}\binom{k+j}{j},
\end{equation*}
where $\Eulerian{i}{j}$ is the Eulerian number, i.e., the number of permutations of $\{1, 2, \cdots, n\}$ having $k$ permutation ascents \cite{Graham1994}.

\section{Conclusion}

We proposed sum rules for permutations with $k$ fixed points, involving Stirling numbers of the first kind. The corresponding identities are partial sums of the moments of these partitions. As an additional benefit, combining a representation of $k^i$ with an expression of Stirling numbers due to Schl\"omlich, a method to derive binomial sums was explained. Several unusual expressions of Bell numbers were also given. In the future, we plan to investigate the specific case of involutions of $\left\{1,2,\cdots,n\right\}$ with $k$ fixed points.

\appendix

\section{Bounds and asymptotic forms}\label{appA}

Adell \cite{Adell2022} obtained the following upper bound
\begin{equation*}
    |s(n+1,m+1)|\leq\frac{n!(\log n)^m}{m!}\left(1+\frac{m}{\log n}\right)
\end{equation*}
and thus, we get
\begin{equation*}
    n!\leq\sum_{k=0}^n\sum_{i=0}^{r+1}|s(r+1,i)|\,k^i\,p_n(k)\leq\sum_{k=0}^n\lambda_{r,k}\,p_n(k)
\end{equation*}
with
\begin{equation*}
    \lambda_{r,k}=\sum_{i=0}^{r+1}\frac{r!(\log r)^{i-1}}{(i-1)!}\left(1+\frac{(i-1)}{\log r}\right)\,k^i
\end{equation*}
which turns out to be equal to
\begin{equation*}
    \lambda_{r,k}=k^2(k\,\log r)^r\left[-1+(1+k)r^k\log r\,E_{-r}(k\log r)\right],
\end{equation*}
where
\begin{equation*}
    E_n(z)=\int_1^{\infty}\frac{e^{-zt}}{t^n}\,\mathrm{d}t
\end{equation*}
is an exponential integral. We have also \cite{deBruijn1981,Lovasz1993,Grunwald2025}:
\begin{equation*}
    B_n\asymp \frac{1}{\sqrt{n}}\left(\frac{n}{W(n)}\right)^{n+1/2}\,\exp\left(\frac{n}{W(n)}-n-1\right)
\end{equation*}
as well as \cite{Odlyzko1995}:
\begin{equation*}
    B_n\asymp \frac{n!}{\sqrt{2\pi W^2(n)\,e^{W(n)}}}\frac{e^{e^{W(n)}-1}}{W^n(n)}
\end{equation*}
where $W$ is the usual Lambert function. This yields $\forall n>0$ \cite{Berend2010}:
\begin{equation*}
    B_n<\left(\frac{0.792\,n}{\log(n+1)}\right)^n.
\end{equation*}

\section{Properties of Stirling numbers of the first kind}\label{appB}

The coefficients arising in the polynomial expansion of
\begin{equation*}
    x(x+1)\cdots(x+(n-1)).
\end{equation*}
are precisely the unsigned Stirling numbers of the first kind One readily sees that these quantities are intimately connected to sums of powers of integers:
\begin{align*}
    x(x+1)\cdots(x+(n-1))=&x^n+x^{n-1}\left(\sum_{k=1}^n(k-1)\right)+x^{n-2}\left(\sum_{i<j=1}^n(i-1)(j-1)\right)+\cdots\nonumber\\
    &+x\left(\prod_{k=1}^n(k-1)\right),
\end{align*}
and it is convenient to write
\begin{equation*}
    x(x+1)\cdots(x+(n-1))=(x+\underbrace{1+\cdots+1}_{n-1 \text{ times}})\cdots(x+1+1)(x+1)x.
\end{equation*}
Expanding this product via the distributive law yields exactly $n!$ terms. Indeed, in the first factor $(x+1+\cdots+1)$ there are $n$ possible choices, in the second factor there are $n-1$, and this pattern continues until the final factor $x$, from which there is only a single choice. More generally, at the $k$th factor there are $n-k+1$ available terms, and precisely one of these contributes an additional power of $x$ to the resulting monomial.

We now describe a parallel construction for generating a permutation $\sigma$. Begin by selecting $\sigma(1)$ from the $n$ possible values. Next, choose $\sigma(\sigma(1))$, then $\sigma(\sigma(\sigma(1)))$, and continue in this manner until a cycle is completed. Once this cycle is closed, select $\sigma(s)$, where $s$ is the smallest element not yet assigned, and repeat the procedure. At the $k$th step of this construction, there are again $n-k+1$ possible choices, exactly one of which results in the formation of a new cycle.

These two procedures are therefore equivalent. Consequently, the number of permutations of $n$ elements with exactly $k$ cycles is given by the coefficient of the monomial $x^k$ in the expansion of $(x+1+\cdots+1)(x+1)x$.

\section{Proof of the binomial expression of $k^i$}\label{appC}

\begin{align*}
    \left[(1+x)^k-x^k\right]^i=&\sum_{l=0}^i(-1)^l\binom{i}{l}x^{kl}(1+x)^{k(i-l)}\\    
    =&\sum_{l=0}^i(-1)^l\binom{i}{l}x^{kl}\sum_{j=0}^{k(i-l)}\binom{k(i-l)}{j}x^{j}\\
    =&\sum_{l=0}^i(-1)^l\binom{i}{l}\sum_{j=0}^{k(i-l)}\binom{k(i-l)}{j}x^{j+kl}.
\end{align*}
The coefficient of $x^{(k-1)i}$ is obtained for $j+kl=(k-1)i$ with $j\geq 0$, i.e., $(k-1)i-kl\geq 0$ or $l\leq\lfloor(k-1)i/k\rfloor$. It therefore follows that it is equal to
\begin{equation*}
    \sum_{l=0}^{\lfloor(k-1)i/k\rfloor}(-1)^l\binom{i}{l}\binom{k(i-l)}{i}.
\end{equation*}
One has also
\begin{equation*}
    \left[(1+x)^k-x^k\right]^i=\left[1+\binom{k}{1}x+\binom{k}{2}x^2+\cdots+\binom{k}{k-1}x^{k-1}\right]^i
\end{equation*}
in which the coefficient of $x^{(k-1)i}$ is
\begin{equation}
    \binom{k}{k-1}^i=k^i,
\end{equation}
and thus
\begin{equation*}
    k^i=\sum_{l=0}^{\lfloor(k-1)i/k\rfloor}(-1)^l\binom{i}{l}\binom{k(i-l)}{i}.
\end{equation*}

\section*{Acknowledgement} 
I would like to thank Eric Fusy for helful comments.


\end{document}